\documentclass{amsart}

\usepackage{graphicx}

\newtheorem{theorem}{Theorem}[section]
\newtheorem{lemma}[theorem]{Lemma}

\theoremstyle{definition}

\theoremstyle{remark}
\newtheorem{remark}[theorem]{Remark}

\numberwithin{equation}{section}

\begin{document}

\title{On the proximity of multiplicative functions with the function counting the number of prime factors with multiplicity}

\author{Theophilus Agama}
\address{Department of Mathematics, African Institute for Mathematical science, Ghana
}
\email{Theophilus@aims.edu.gh/emperordagama@yahoo.com}



\date{\today.}



\begin{abstract}
Given an additive function $f$ and a multiplicative function $g$, we set $E(f,g;x)=\# \{n\leq x:f(n)=g(n)\}$. We investigate the size of this quantity; in particular, we establish   lower bounds for $E(\Omega , g;x)$, where $\Omega(n)$ stands for the number of prime factors of $n$ counting their multiplicity and where $g$ is an arbitrary multiplicative function. We show that $\displaystyle{E(\Omega, g, x)\gg \frac{x}{(\log \log x)^{\frac{1}{2}+\epsilon}}}$, for any arbitrarily small $\epsilon >0$.  This is therefore an extension of an earlier result of Dekonick, Doyon and Letendre.
\end{abstract}

\maketitle

\section{Introduction}
Let us set $E(f,g;x):= \# \{n\leq x:f(n)=g(n)\}$, where $f$ and $g$ are arbitrary additive and multiplicative functions, respectively. One of the basic questions one can ever ask is, how large and how small can this quantity be. In 2014, Dekonick, Doyon and Letendre \cite{de2014proximity} proved that for some suitable choice of multiplicative function and some choice of sequence $(x_n)$ of positive integers \begin{align}E(\omega ,g;x_n)\gg \frac{x_n}{(\log \log x_n)^{\frac{1}{2}+\epsilon}},\nonumber
\end{align}for any small $\epsilon >0$. Above all they were able to show that if $f$ is an integer-valued additive function such that \begin{align}\varphi (x)=\varphi_f(x)=\frac{B(x)}{A(x)}\longrightarrow 0\nonumber
\end{align}as $x\longrightarrow \infty$, where \begin{align*}A(x):=\sum \limits_{p^{\alpha}\leq x}f(p^{\alpha})\bigg(1-\frac{1}{p}\bigg) \quad \text{~and~} \quad B(x):=\sum \limits_{p^{\alpha}\leq x}\frac{|f(p^{\alpha})|^2}{p^{\alpha}},
\end{align*}and that\begin{align}\max_{z\in \mathbb{R}} \# \{n\leq x :f(n)=z\}=O\bigg(\frac{x}{K(x)}\bigg),\nonumber
\end{align}where $K(x)\longrightarrow \infty$ as $x\longrightarrow \infty$. Then, for any multiplicative function $g$ \begin{align}E(f,g;x)=o(x) \nonumber
\end{align}as $x\longrightarrow \infty$. In particular, $E(\omega,g,x)=o(x)$ as $x\longrightarrow \infty$.

\footnote{On the proximity of multiplicative functions to the function counting the number of prime divisors with multiplicity
\par
}%
.

\section{Prelimary results}

\begin{lemma}\label{llady}
Let $\pi_k(x)=\# \{n\leq x:\omega (n)=k\}$ for each positive integer $k$. Then the maximum value of $\pi_k (x)$ is $\frac{x}{\sqrt{\log \log x}}(1+o(1))$ and the value of $k$ for which it occurs is $k=\log \log x+O(1)$
\end{lemma}

\begin{proof}
This follows from a result of Balazard \cite{balazard1990unimodalite}.
\end{proof}

\begin{lemma}\label{normal}
For all $x\geq 2$ and for every $\delta >0$ \begin{align*}\# \{n\leq x: |\omega (n)-\log \log n|>(\log \log x)^{1+\delta}\}=o(x) \quad (x\longrightarrow \infty).
\end{align*}
\end{lemma}

\begin{proof}
This follows from Theorem 8.12 in the book of Nathason \cite{nathanson2000elementary}.
\end{proof}

\begin{remark}
Now we present one of the results and the techniques of Dekoninck, Doyon and Letendre \cite{de2014proximity} employed in obtaining the lower bound for the quantity $E(\omega, g, x)$.
\end{remark}

\begin{theorem}
Let $\epsilon>0$ be very small. Then, there exist a multiplicative function $g$ and a sequence $(r_j)$ of positive integers such that \begin{align*}E(\omega ,g;r_j)\gg \frac{x}{(\log \log r_j)^{\frac{1}{2}+\epsilon}}. 
\end{align*}
\end{theorem}

\begin{proof}
Given $\epsilon >0$ very small,  let $\mathcal{S}=\{s_1, s_2, \ldots\}$ be an infinite set of primes, with  $s_1=2$ and  $s_j$ to be the smallest prime number larger than \begin{align}\max \{s_{j-1},j^{1+\delta}\}\nonumber,
\end{align}for $j\geq 2$ and $\delta >0$ very small. Choose $r_j=e^{e^{2^{j}}}$ and let $(z_j)$ be sequence of integers maximizing the quantity\begin{align}\# \left \{m \leq \frac{r_j}{s_j}: s_k\not| m \text{~for~each~}s_k \in \mathcal{S}, \omega(m)=z_j-1\right \},
\end{align}for each $j\geq 1$, which is well defined by Lemma \ref{llady}. Define  $g$ a strongly multiplicative function on the primes as \begin{align*}g(p)=\begin{cases}z_j & \text{~if~}\quad p=s_j\in \mathcal{S}\\ 1 & \text{~if~}\quad p\notin \mathcal{S}.\end{cases}
\end{align*}To find a lower bound for $E(\omega, g;r_j)$, it suffices to consider integers of the form $n=m\cdot s_j$ such that $s_j\not| m$ for $s_j\in \mathcal{S}.$
\begin{align*}E(\omega, g;r_j)&=\# \{n\leq r_j:\omega (n)=g(n)\}\\&\geq \# \{n\leq r_j:s_j|n,\quad s_k\not| m \quad \text{for~}k\neq j,\quad \omega(m)=z_j\}\\& \geq \# \{m\leq \frac{r_j}{s_j}:s_k\not| m \text{~for~} s_k\in \mathcal{S}, ~\omega (m)=z_j-1\}.
\end{align*}Now, let $I_{j}=[\log \log r_j-(\log \log r_j)^{\frac{1}{2}+\epsilon}, \log \log r_j+(\log \log r_j)^{\frac{1}{2}+\epsilon}]$. Then\begin{align*}\# \{m\leq \frac{r_j}{s_j}:s_k\not| m \text{~for~}s_k\in \mathcal{S}\}=\# \{m\leq \frac{r_j}{s_j}:s_k\not| m \text{~for~}s_k\in S, \omega (m)\notin I_j\}\\+\# \{m\leq \frac{r_j}{s_j}:s_k\not| m,  \omega(m)\in I_j\}.
\end{align*}In relation to Lemma \ref{normal} \begin{align}\label{j}\# \{m\leq \frac{r_j}{s_j}:s_k\not| m \text{~for~}s_k\in S, ~\omega (m)\notin I_j\}=o\bigg(\frac{r_j}{s_j}\bigg) ~~(j\longrightarrow \infty).
\end{align}Also \begin{align}\label{y}\# \{m\leq \frac{r_j}{s_j}:s_k\not| m, ~\omega(m)\in I_j\}&\leq \sum \limits_{N\in I_{j}}\# \{m\leq \frac{r_j}{s_j}:s_k\not| m \text{~for~}s_k\in \mathcal{S},~\omega(m)=N\}\nonumber\\& \leq 2(\log \log r_j)^{\frac{1}{2}+\epsilon}\# \{m\leq \frac{r_j}{s_j}:\text{~for~}s_k \in S,~\omega(m)=z_j-1\}. 
\end{align}It follows from \eqref{y} and \eqref{j} that \begin{align}\label{T}\# \{m\leq \frac{r_j}{s_j}:s_k\not| m \text{~for~}s_{k}\in S, \omega(m)=z_j-1\}\geq \frac{\# \{m\leq \frac{r_j}{s_j}:s_k\not| m \text{~for~}s_k\in S\}}{2(\log \log r_j )^{\frac{1}{2}+\epsilon}}.
\end{align}It is also clear that\begin{align}\label{y1}\# \{m\leq \frac{r_j}{s_j}:s_k\not| m \text{~for~}s_k\in S\}=\bigg(1+o(1)\bigg)\frac{r_j}{s_j}C(\delta) ~~  (j\longrightarrow \infty).
\end{align}Plugging \eqref{y1} into \eqref{T}, we have that \begin{align}\# \{m\leq \frac{r_j}{s_j}:s_k\not| m \text{~for~}s_{k}\in S, \omega(m)=z_j-1\}\gg \bigg(\frac{1}{2}+o(1)\bigg)\frac{r_j}{s_j}C(\delta)\cdot \frac{1}{(\log \log r_j)^{\frac{1}{2}+\epsilon}}.
\end{align}Then, for sufficiently large $j$ we see that , $s_j<j^{1+2\delta}\leq (2^{j})^{\epsilon}$, where $(2^{j})^{\epsilon} =(\log \log r_j)^{\epsilon}$. Using this fact,  we obtain\begin{align}\label{L}\bigg(\frac{1}{2}+o(1)\bigg)\frac{r_j}{s_j}C(\delta)\cdot \frac{1}{(\log \log r_j)^{\frac{1}{2}+\epsilon}} \gg \bigg(\frac{1}{2}+o(1)\bigg)\frac{r_j}{(\log \log r_j)^{\frac{1}{2}+2\epsilon}},
\end{align}and from \eqref{L} we have \begin{align*}E(\omega, g;r_j)\gg \frac{r_j}{(\log \log r_j)^{\frac{1}{2}+\epsilon}},
\end{align*}thus completing the proof.
\end{proof}

\begin{remark}
It has to be said that this is a good lower bound, but it only works for a particular type of sequence and therefore is not uniform.
\end{remark}

\section{Main result}
In this section we use the techniques employed by Dekonick, Doyon and Letendre \cite{de2014proximity} to obtain a uniform lower bound for $E(\Omega, g,x)$.

\begin{theorem}
Let $\mathcal{S}=\{s_1,s_2,\ldots\}$ be an infinite set of the primes such that \begin{align*}\sum \limits_{j=1}^{\infty}\frac{1}{s_j}<\infty.
\end{align*}Then for any small $\epsilon >0$, there exists a strongly multiplicative function $g$ such that  \begin{align*}E(\Omega, g,x)\gg \frac{x}{(\log \log x)^{\frac{1}{2}+\epsilon}}.
\end{align*}
\end{theorem} 

\begin{proof}
Let $(z_j)$ be sequence of positive integers maximizing the quantity\begin{align*}\# \{r\leq \frac{x}{s_j}:s_i\not| r \quad \text{~for~each~}s_i\in \mathcal{S},~~\Omega (r)=z_j-2\},
\end{align*}for $s_j\equiv 1\pmod 4$ for each $j\geq 1$,  and\begin{align*}\# \{r\leq \frac{x}{s_j}:s_i\not| r \quad \text{~for~each~}s_i\in \mathcal{S},~~\Omega (r)=z_j\},
\end{align*}for $s_j\equiv 3\pmod 4$ for each $j\geq 1$, which is well defined by Lemma \ref{llady}.
Define $g$, a strongly multiplicative function, on the primes as\begin{align*}g(p)=\begin{cases}z_j-1 \quad \text{~if~}p=s_j\equiv 1\pmod 4,~ s_j\in \mathcal{S}\\z_j+1 \quad \text{~if~}p=s_j\equiv 3 \pmod 4,~ s_j\in \mathcal{S}\\1 \quad \text{~if~} \quad p\notin \mathcal{S}.\end{cases}
\end{align*}To obtain a lower bound for $E(\Omega, g,x)$, it suffices to consider only integers of the form $n=r\cdot s_j^{\alpha}$ for $\alpha \geq 1$, $s_j\not| r$. Clearly \begin{align*}E(\Omega, g,x)&=\# \{n\leq x:\Omega (n)=g(n)\}\\&\geq \sum \limits_{\alpha \geq 1}\# \{n\leq x:~s_j^{\alpha}|n,\quad s_i\not| n \quad \text{~for~}\quad i\neq j,\quad \Omega(n)=g(n)\}\\&  \geq \sum \limits_{\alpha \geq 1}\# \{r\leq \frac{x}{s_j^{\alpha}}: \Omega (r)=g(s_j)-\alpha,~ s_i\not| r\text{~for~each~}s_i\in \mathcal{S}, \quad s_j\equiv 1 \pmod 4\}\\&\geq \# \{r\leq \frac{x}{s_j}:\Omega (r)=g(s_j)-1,~ s_i\not| r\text{~for~each~}s_i\in \mathcal{S},\quad s_j\equiv 1 \pmod 4\}\\& \geq \# \{r\leq \frac{x}{s_j}:\Omega (r)=z_j-2,~ s_i\not| r\text{~for~each~}s_i\in \mathcal{S}\}.
\end{align*}Again consider the interval\begin{align*}I=[\log \log x-(\log \log x)^{\frac{1}{2}+\epsilon}, \log \log x+(\log \log x)^{\frac{1}{2}+\epsilon}].
\end{align*}Let us  consider \begin{align*}\# \{r\leq \frac{x}{s_j}:s_i\not| r\text{~for~each~}s_i\in \mathcal{S},~s_j\equiv 1 \pmod 4\}.
\end{align*}We observe, in relation to Theorem \ref{normal},  $\# \{r\leq \frac{x}{s_j}:s_i\not| r\text{~for~each~}s_i\in \mathcal{S},~s_j\equiv 1 \pmod 4,~\Omega (r)\notin I_j\}=o\bigg(\frac{x}{s_j}\bigg)$, as $j\longrightarrow \infty$. On the other hand\begin{align*}\# \{r\leq \frac{x}{s_j}:s_i\not| r\quad \text{~for~each~}\quad s_i\in \mathcal{S},~s_j\equiv 1 \pmod 4,~\Omega (r)\in I_j\}\end{align*}\begin{align}= \sum \limits_{\mathcal{U}\in I_j}\# \{r\leq \frac{x}{s_j}:s_i\not| r\quad \text{~for~each~}\quad s_i\in \mathcal{S}, ~\Omega (r)=\mathcal{U}, \quad  s_j \equiv 1 \pmod 4\}\nonumber\end{align}\begin{align}\leq 2(\log \log x)^{\frac{1}{2}+\epsilon}\# \{r\leq \frac{x}{s_j}:s_i\not| r\text{~for~each~}s_i\in \mathcal{S},\quad  \Omega(r)=z_j-2\}.\label{my4}
\end{align}It follows from \eqref{my4}, that \begin{align*}\# \{r\leq \frac{x}{s_j}:s_i\not| r\text{~for~each~}s_i\in \mathcal{S},~\Omega (r)=z_j-2\}\end{align*}\begin{align*}&\geq \frac{1}{2(\log \log x)^{\frac{1}{2}+\epsilon}}\# \{r\leq \frac{x}{s_j}:s_i\not| r\text{~for~each~}s_i\in \mathcal{S},~s_j\equiv 1 \pmod 4\}\\&\geq \frac{1}{2(\log \log x)^{\frac{1}{2}+\epsilon}} \sum \limits_{s_j\equiv 1\pmod 4}\# \{r\leq \frac{x}{s_j}:s_i\not| r\text{~for~each~}s_i\in \mathcal{S}\}\\&\geq \frac{1}{2(\log \log x)^{\frac{1}{2}+\epsilon}} \sum \limits_{s_j\equiv 1\pmod 4}\frac{x}{s_j}(1+o(1))C(\mathcal{S})\\&\geq  \frac{1}{2(\log \log x)^{\frac{1}{2}+\epsilon}}x(1+o(1))C(\mathcal{S})\sum \limits_{{s_j\equiv 1\pmod 4}}\frac{1}{s_j}\\&\geq \frac{1}{2(\log \log x)^{\frac{1}{2}+\epsilon}}x(1+o(1))C(\mathcal{S})\mathrm{K},\label{r}
\end{align*}for some positive real number  $\mathrm{K}$  and \begin{align*}C(\mathcal{S})=\prod \limits_{j=1}^{\infty}\bigg(1-\frac{1}{s_j}\bigg) \quad \text{~and~}\sum \limits_{s_j\equiv 1\pmod 4}\frac{1}{s_j}<\infty.
\end{align*}Carrying out the same process for the other residue class $s_j\equiv 3\pmod 4$ and combining the result, we will obtain\begin{align*}E(\Omega ,g,x)\gg \frac{x}{(\log \log x)^{\frac{1}{2}+\epsilon}}.
\end{align*}
\end{proof}

\section{Conclusion}
The lower bound obtained in the original work of Dekoninck, Doyon and Letendre \cite{de2014proximity} can be made uniform by using a similar choice of multiplicative function in the main result; that is, if we let \begin{align*}g(p)=\begin{cases}z_j-1 \quad \text{~if~}p=s_j\equiv 1\pmod 4,~ s_j\in \mathcal{S}\\z_j+1 \quad \text{~if~}p=s_j\equiv 3 \pmod 4,~ s_j\in \mathcal{S}\\1 \quad \text{~if~} \quad p\notin \mathcal{S}.\end{cases}
\end{align*}Then \begin{align*}E(\omega, g,x)\gg \frac{x}{(\log \log x)^{\frac{1}{2}+\epsilon}}
\end{align*}holds uniformly.

\bibliographystyle{amsplain}

\begin{thebibliography}{10}

\bibitem {de2014proximity} J.M De Koninck, N. Doyon and Letendre, \textit{On the proximity of additive and multiplicative functions.}, Functiones et Approximatio Commentarii Mathematici, vol. 52.2,
Adam Mickiewicz University, 2015, pp.327--344.


\bibitem {balazard1990unimodalite}Balazard, Michel, \textit{Unimodalit{\'e} de la distribution du nombre de diviseurs premiers d’un entier},
Ann. Inst. Fourier, Grenoble, vol. 40.2, 1990, pp.255--270.

\bibitem {nathanson2000elementary}Nathanson, M.B, \textit{Graduate Texts in Mathematics},
New York, NY: Springer New York, 2000.


\end{thebibliography}

\end{document}